\documentclass[12pt,reqno]{article}
\usepackage{amsfonts,amsmath,amsthm,amssymb,latexsym}
\usepackage{graphicx}
\date{}
\newcommand{\R}{{\mathbb R}}
\newcommand{\D}{{\mathcal D}}
\newcommand{\const}{\mathrm{const}}
\newcommand{\esslim}{\operatornamewithlimits{ess\,lim}}
\newcommand{\sign}{\operatorname{sign}}
\newcommand{\Int}{\operatorname{Int}}
\newtheorem{theorem}{Theorem}
\newtheorem{corollary}{Corollary}
\newtheorem{lemma}{Lemma}
\newtheorem{proposition}{Proposition}
\theoremstyle{definition}
\newtheorem{definition}{Definition}
\newtheorem{remark}{Remark}

\title{On the structure of entropy solutions to the Riemann problem for a degenerate nonlinear parabolic equation}

\author{Evgeny~Yu.~Panov
\\ \small
Yaroslav-the-Wise Novgorod State University, Veliky Novgorod, Russian Federation, \\ \small Research and Development Center, Veliky Novgorod, Russian Federation. }

\begin{document}
\maketitle

\begin{abstract}
We find an explicit form of entropy solutions to a Riemann problem for a degenerate nonlinear parabolic equation with piecewise constant velocity and diffusion coefficients. It is demonstrated that this solution corresponds to the minimum point of some strictly convex function of a finite number of variables.
\end{abstract}

\section{Introduction}
In a half-plane $\Pi=\{ (t,x) \ | \ t>0, x\in\R \}$, we consider a nonlinear parabolic equation
\begin{equation}\label{1}
u_t+v(u)u_x-t(a^2(u)u_x)_x=0,
\end{equation}
where $v(u),a(u)\in L^\infty(\R)$, $a(u)\ge 0$. Since the diffusion coefficient $a(u)$ may take zero value, equation (\ref{1}) is degenerate. In the case when $a(u)\equiv 0$ it reduces to a first order conservation law
\begin{equation}\label{con}
u_t+\varphi(u)_x=0,
\end{equation}
where $\varphi'(u)=v(u)$. Similarly, a general equation (\ref{1}) can be written in the conservative form
$$
u_t+\varphi(u)_x-tA(u)_{xx}=0
$$
with $A'(u)=a^2(u)$, which allows to define weak solutions of this equation. Unfortunately, weak solutions to a Cauchy problem for equation (\ref{1}) are not unique in general, and some additional entropy conditions are required. We consider the Cauchy problem with initial data
\begin{equation}\label{2}
u(0,x)=u_0(x),
\end{equation}
where $u_0(x)\in L^\infty(\R)$. Recall the notion of entropy solution (e.s. for short) in the sense of Carrillo \cite{Car}.

\begin{definition}\label{def1}
A function $u=u(t,x)\in L^\infty(\Pi)$ is called an e.s. of (\ref{1}), (\ref{2}) if

(i) the distribution $A(u)\in L^2_{loc}(\Pi)$;

(ii) for all $k\in\R$
\begin{equation}\label{en}
|u-k|_t+(\sign(u-k)(\varphi(u)-\varphi(k)))_x-(t\sign(u-k)(A(u)-A(k)))_{xx}\le 0
\end{equation}
in the sense of distributions (in $\D'(\Pi)$);

(iii) $\displaystyle\esslim_{t\to 0} u(t,\cdot)=u_0$ in $L^1_{loc}(\R)$.
\end{definition}

Entropy condition (\ref{en}) means that for each nonnegative test function $f=f(t,x)\in C_0^2(\Pi)$
\begin{equation}\label{eni}
\int_{\Pi}[|u-k|f_t+\sign(u-k)((\varphi(u)-\varphi(k))f_x+t(A(u)-A(k))f_{xx})]dtdx\ge 0.
\end{equation}
In the case of conservation laws (\ref{con}) the notion of e.s. reduces to the notion of generalized e.s. in the sense of Kruzhkov \cite{Kr}.
Taking in (\ref{en}) $k=\pm M$, $M\ge\|u\|_\infty$, we derive that
$$
u_t+\varphi(u)_x-tA(u)_{xx}=0 \ \mbox{ in } \D'(\Pi),
$$
that is, an e.s. $u$ of (\ref{1}), (\ref{2}) is a weak solution of this problem. It is known that an e.s. of (\ref{1}), (\ref{2}) always exists and is unique. In general multidimensional setting this was demonstrated in \cite{Kr} for conservation laws and in \cite{Car} for the general case. If to be precise, in \cite{Car} the case of usual diffusion term $A(u)_{xx}$ was studied but the proofs can be readily adapted to the case of the self-similar diffusion $tA(u)_{xx}$.

If $u=u(t,x)$ is a piecewise $C^2$-smooth e.s. of equation (\ref{1}) then it must satisfy this equation in classic sense in each smoothness domain. Applying relation (\ref{1}) to a test function $f=f(t,x)\in C_0^2(\Pi)$ supported in a neighborhood of a discontinuity line $x=x(t)$ and integrating by parts, we then obtain the identity
\begin{equation}\label{disc}
(-x'(t)[u]+[\varphi(u)]-t[A(u)_x])f+t[A(u)]f_x=0
\end{equation}
a.e. on the line $x=x(t)$. Here we denote by $[w]$ the jump of a function $w=w(t,x)$ on the line $x=x(t)$ so that
$$
[w]=w(t,x(t)+)-w(t,x(t)-), \quad \mbox{ where } w(t,x(t)\pm)=\lim_{y\to x(t)\pm} w(t,y).
$$
Since the functions $f$, $f_x$ are arbitrary and independent on the line $x=x(t)$, identity (\ref{disc}) implies the following two relations of Rankine-Hugoniot type
\begin{align}\label{RG1}
[A(u)]=0, \\
\label{RG}
-x'(t)[u]+[\varphi(u)]-t[A(u)_x]=0.
\end{align}
Similarly, it follows from entropy relation (\ref{eni}), after integration by parts, that
\begin{align}\label{disc1}
(-x'(t)[|u-k|]+[\sign(u-k)(\varphi(u)-\varphi(k))]-t[\sign(u-k)A(u)_x])f+ \nonumber\\ t[\sign(u-k)(A(u)-A(k))]f_x\le 0.
\end{align}
Since the function $A(u)$ increases, it follows from (\ref{RG1}) that $A(u)=\const$ when $u$ lies between the values $u(t,x(t)-)$ and $u(t,x(t)+)$. This implies that $[\sign(u-k)(A(u)-A(k))]=0$ and in view of arbitrariness of $f\ge 0$ it follows from (\ref{disc1}) that
\begin{equation}\label{disc2}
-x'(t)[|u-k|]+[\sign(u-k)(\varphi(u)-\varphi(k))]-t[\sign(u-k)A(u)_x]\le 0.
\end{equation}
In the case when $k$ lies out of the interval with the endpoints $u^\pm\doteq u(t,x(t)\pm)$ relation (\ref{disc2})
follows from (\ref{RG}) and fulfils with equality sign. When $u^-<k<u^+$ this relation reads
$$
-x'(t)(u^++u^--2k)+\varphi(u^+)+\varphi(u^-)-2\varphi(k)-t(A(u)_x^++A(u)_x^-)\le 0,
$$
where $A(u)_x^\pm=A(u)_x(t,x(t)\pm)$. Adding (\ref{RG}) to this relation and dividing the result by $2$, we arrive at the following analogue of the famous Oleinik condition (see \cite{Ol}) known for conservation laws.
\begin{equation}\label{Ol}
-x'(t)(u^+-k)+\varphi(u^+)-tA(u)_x^+-\varphi(k)\le 0 \quad \forall k\in [u^-,u^+].
\end{equation}
In the case $u^+<u^-$ this condition has the form
\begin{equation}\label{Ol-}
-x'(t)(u^+-k)+\varphi(u^+)-tA(u)_x^+-\varphi(k)\ge 0 \quad \forall k\in [u^+,u^-]
\end{equation}
and can be derived similarly.
Geometric interpretation of these conditions is that the graph of the flux function $\varphi(u)$ lies not below (not above) of the segment connecting the points $(u^-,\varphi(u^-)-tA(u)_x^-)$, $(u^+,\varphi(u^+)-tA(u)_x^+)$ when $u^-\le u\le u^+$ (respectively, when $u^+\le u\le u^-$), see Figure~\ref{fig1}. We take here into account that in view of condition (\ref{RG}) the vector $(-x'(t),1)$ is a normal to the indicated segment. We also notice that it follows from relations (\ref{Ol}), (\ref{Ol-}) with $k=u^\pm$ and from the Rankine-Hugoniot condition (\ref{RG}) that $A(u)_x^\pm\ge 0$ ($A(u)_x^\pm\le 0$) whenever $u^+>u^-$ ($u^+<u^-$).

\begin{figure}[h!]
\centering
\includegraphics[width=4in]{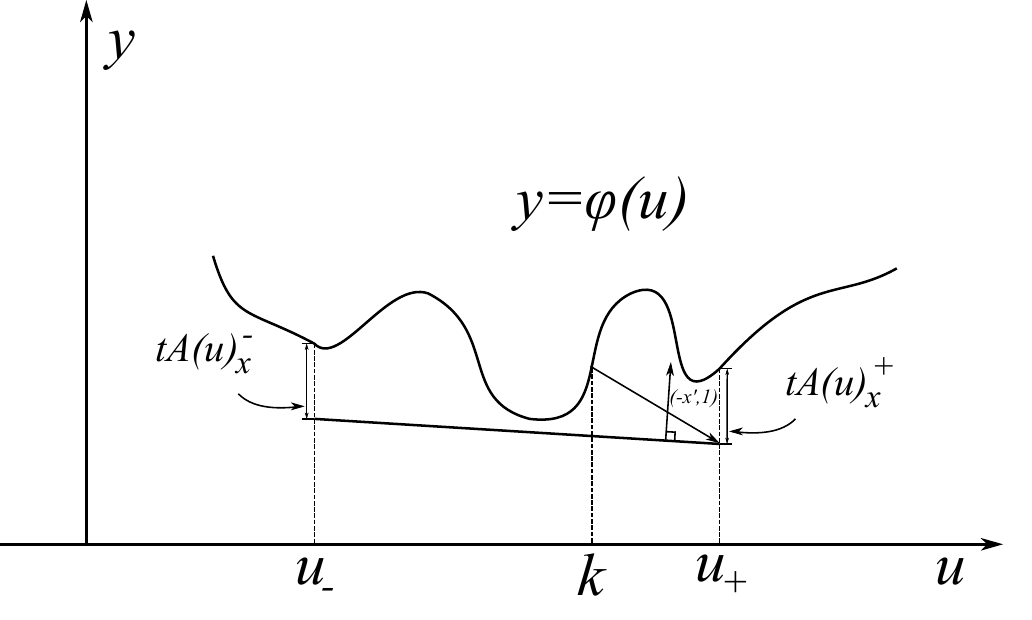}
\caption{Oleinik condition.
\label{fig1}}
\end{figure}

\section{The case of piecewise constant coefficients.}

Below we will assume that the functions $v(u)$, $a(u)$ are piecewise constant, $v(u)=v_k$, $a(u)=a_k$ when
$u_k<u<u_{k+1}$, $k=0,\ldots,n-1$, where $$\alpha=u_0<u_1<\cdots<u_{n-1}<u_n=\beta.$$ We will study problem (\ref{1}), (\ref{2}) with the Riemann data $u_0(x)=\left\{\begin{array}{lr} \alpha, & x<0, \\ \beta, & x>0.\end{array}\right.$ Since this problem is invariant
under the scaling transformations $t\to \lambda t$, $x\to\lambda x$, $\lambda>0$ then, by the uniqueness, the e.s.  $u=u(t,x)$ is self-similar: $u(t,x)=u(\lambda t,\lambda x)$. This implies that $u=u(x/t)$. Suppose that $a_k>0$. Then in a domain where $u_k<u(\xi)<u_{k+1}$ with $\xi=x/t$ equation (\ref{1}) reduces to the second order ODE
$$
(v_k-\xi) u'-a_k^2 u''=0,
$$
the general solution of which is $u=C_1 F((\xi-v_k)/a_k)+C_2$; $C_1,C_2=\const$, where $$F(z)=\frac{1}{\sqrt{2\pi}}\int_{-\infty}^z e^{-s^2/2}ds$$ is the error function.
Therefore, it is natural to seek the e.s. of our problem in the following form
\begin{align}\label{sol}
u(\xi)=\left\{\begin{array}{lcr}
u_k+\frac{u_{k+1}-u_k}{F((\xi_{k+1}-v_k)/a_k)-F((\xi_k-v_k)/a_k)}(F((\xi-v_k)/a_k)-F((\xi_k-v_k)/a_k)) & , & a_k>0, \\
u_k & , & a_k=0, \end{array}\right. \\ \nonumber
\xi_k<\xi<\xi_{k+1}, \ k=0,\ldots,d,
\end{align}
where
$$d=\left\{\begin{array}{lcr} n-1 & , & a_{n-1}>0, \\ n & , & a_{n-1}=0, \end{array}\right. \quad -\infty=\xi_0<\xi_1\le \cdots\le \xi_d<\xi_{d+1}=+\infty,$$ and we agree that $a_n=0$, $F(-\infty)=0$, $F(+\infty)=1$.
We also assume that $\xi_{k+1}>\xi_k$ whenever $a_k>0$.
The rays $x=\xi_kt$ for finite $\xi_k$ are (weak or strong) discontinuity lines of $u$, they correspond to discontinuity points $\xi_k$ of the function $u(\xi)$. Observe that conditions (\ref{RG1}), (\ref{RG}) turns into the following relations at points $\xi_k$
\begin{align}\label{RG1a}
[A(u)]=A(u(\xi_k+))-A(u(\xi_k-))=0, \\
\label{RGa}
-\xi_k [u]+[\varphi(u)]-[A(u)']=-\xi_k(u(\xi_k+)-u(\xi_k-))+\varphi(u(\xi_k+))-\nonumber\\ \varphi(u(\xi_k-))-A(u)'(\xi_k+)+A(u)'(\xi_k-)=0.
\end{align}
Here $w(\xi_k\pm)$ denotes unilateral limits of a function $w(\xi)$ at the point $\xi_k$. Similarly, the Oleinik condition
(\ref{Ol}) reads
\begin{equation}\label{Ola}
-\xi_k(u(\xi_k+)-k)+\varphi(u(\xi_k+))-A(u)'(\xi_k+)-\varphi(k)\le 0 \quad \forall k\in [u(\xi_k-),u(\xi_k+)].
\end{equation}
Notice that our solution (\ref{sol}) is a nonstrictly increasing function of the self-similar variable $\xi$ and, therefore, $u(\xi_k-)\le u(\xi_k+)$.

Let us firstly analyze the solution (\ref{sol}) in the case $\xi_{k-1}<\xi_k<\xi_{k+1}$. If $a_{k-1},a_k>0$ then $u(\xi_k-)=u(\xi_k+)=u_k$ so that condition (\ref{RG1a}) fulfils while (\ref{RGa}) reduces to the equality $[A(u)']=0$,
which is revealed as
\begin{align}\label{c++}
\frac{a_k(u_{k+1}-u_k)}{F((\xi_{k+1}-v_k)/a_k)-F((\xi_k-v_k)/a_k)}F'((\xi_k-v_k)/a_k)=\nonumber\\
\frac{a_{k-1}(u_k-u_{k-1})}{F((\xi_k-v_{k-1})/a_{k-1})-F((\xi_{k-1}-v_{k-1})/a_{k-1})}F'((\xi_k-v_{k-1})/a_{k-1}).
\end{align}
In this situation $\xi=\xi_k$ is a weak discontinuity point, the function $u(\xi)$ itself is continuous, only its derivative $u'(\xi)$ may be discontinuous. Moreover, it follows from (\ref{c++}) that both functions $u(\xi)$ and $u'(\xi)$ are continuous at point $\xi_k$ if $a_k=a_{k-1}>0$.

If $a_{k-1}>a_k=0$ then again $u(\xi)$ is continuous at $u_k$ and (\ref{RGa}) reduces to the relation
\begin{equation}\label{c+0}
\frac{a_{k-1}(u_k-u_{k-1})}{F((\xi_k-v_{k-1})/a_{k-1})-F((\xi_{k-1}-v_{k-1})/a_{k-1})}F'((\xi_k-v_{k-1})/a_{k-1})=0,
\end{equation}
which is impossible. If $a_k>a_{k-1}=0$ then $u(\xi_k-)=u_{k-1}<u_k=u(\xi_k+)$, that is, $\xi_k$ is a strong discontinuity point. Condition (\ref{RG1a}) holds because $A(u)$ is constant on $[u_{k-1},u_k]$ ($A'(u)=a_{k-1}^2=0$) while (\ref{RGa}) turns into
\begin{equation}\label{c0+}
(v_{k-1}-\xi_{k})(u_k-u_{k-1})-\frac{a_k(u_{k+1}-u_k)}{F((\xi_{k+1}-v_k)/a_k)-F((\xi_k-v_k)/a_k)}F'((\xi_k-v_k)/a_k)=0,
\end{equation}
where we use the fact that $\varphi(u_k)-\varphi(u_{k-1})=v_{k-1}(u_k-u_{k-1})$. It remains to analyze the situation
when $a_k=a_{k-1}=0$. In this case again $u(\xi_k-)=u_{k-1}<u_k=u(\xi_k+)$ and $A(u_{k-1})=A(u_k)$ while condition (\ref{RGa}) turns into the simple relation
\begin{equation}\label{c00}
\xi_{k}=v_{k-1}.
\end{equation}
Finally, since the function $\varphi(u)$ is affine on the segment $[u_{k-1},u_k]$ and $A(u)'(\xi_k\pm)\ge 0$, then entropy relation (\ref{Ola}) is always satisfied.

Now we consider the case when there exists a nontrivial family of mutually equaled values $\xi_i$, $\xi_i=\xi_k$ for $i=k,\ldots,l$, where $l>k$. We can assume that this family is maximal, that is, $$-\infty\le\xi_{k-1}<\xi_k=\cdots =\xi_l<\xi_{l+1}\le +\infty.$$
Then $a_i=0$ for $i=k,\ldots,l-1$, and the point $\xi=c\doteq\xi_k$ is a discontinuity point of $u(\xi)$ with the unilateral limits
$$u(c+)=u_l, \ u(c-)=u_{k'}, \ \mbox{ where } k'=\left\{\begin{array}{lcr} k & , & a_{k-1}>0, \\ k-1 & , & a_{k-1}=0 \end{array}\right.. $$
Since $a(u)=0$ for $u(c-)<u<u(c+)$, we find that $A(u(c-))=A(u(c+))$ and condition (\ref{RG1a}) is satisfied.
Further, we notice that
\begin{align*}
\sum_{i=k'}^{l-1}(-\xi_{i+1}(u_{i+1}-u_i))=-c\sum_{i=k'}^{l-1}(u_{i+1}-u_i)=-c(u_l-u_{k'}), \\ \sum_{i=k'}^{l-1}v_i(u_{i+1}-u_i)=\sum_{i=k'}^{l-1}(\varphi(u_{i+1})-\varphi(u_i))=\varphi(u_l)-\varphi(u_{k'}).
\end{align*}
Therefore, condition (\ref{RGa}) can be written in the form
\begin{equation}\label{RGb}
\sum_{i=k'}^{l-1}(v_i-\xi_{i+1})(u_{i+1}-u_i)-(A(u)'(c+)-A(u)'(c-))=0,
\end{equation}
where, as is easy to verify,
\begin{align}\label{der-}
A(u)'(c-)=\left\{\begin{array}{lcr} 0 & , & a_{k-1}=0, \\
\frac{a_{k-1}(u_k-u_{k-1})}{F((\xi_k-v_{k-1})/a_{k-1})-F((\xi_{k-1}-v_{k-1})/a_{k-1})}F'((\xi_k-v_{k-1})/a_{k-1}) & , & a_{k-1}>0, \end{array}\right. \\
\label{der+}
A(u)'(c+)=\left\{\begin{array}{lcr} 0 & , & a_l=0, \\
\frac{a_l(u_{l+1}-u_l)}{F((\xi_{l+1}-v_l)/a_l)-F((\xi_l-v_l)/a_l)}F'((\xi_l-v_l)/a_l) & , & a_l>0. \end{array}\right.
\end{align}
In the similar way we can write the Oleinik condition (\ref{Ola}) as follows
\begin{equation}\label{Olb}
\sum_{i=j}^{l-1}(v_i-\xi_{i+1})(u_{i+1}-u_i)-A(u)'(c+)\le 0 \quad k'<j<l.
\end{equation}
We use here the fact the function $\varphi(u)$ is piecewise affine and, therefore, it is enough to verify the Oleinik condition (\ref{Ola}) only at the nodal points $k=u_j$.

The above reasoning remains valid also in the case when $l=k$. In this case, relation (\ref{RGb}) reduces to one of conditions (\ref{c++}), (\ref{c+0}), (\ref{c0+}), (\ref{c00}) while (\ref{Olb}) is trivial.

\section{The entropy function}
We introduce the convex cone
$\Omega\subset\R^d$ consisting of points $\bar\xi=(\xi_1,\ldots,\xi_d)$ with increasing coordinates,
$\xi_1\le\xi_2\le\cdots\le\xi_d$ such that $\xi_{k+1}>\xi_k$ whenever $a_k>0$, $k=1,\ldots,d-1$.
Each point $\bar\xi\in\Omega$ determines a function $u(\xi)$ in correspondence with formula (\ref{sol}).
Assume firstly that $\bar\xi\in\Int\Omega$, that is, the values $\xi_k$ are strictly increasing. Then conditions
(\ref{c++}), (\ref{c+0}), (\ref{c0+}), (\ref{c00}) coincides with the equality $\frac{\partial}{\partial\xi_k} E(\bar\xi)=0$, where
\begin{align}\label{E}
E(\bar\xi)=-\sum_{k=0,\ldots,n-1,a_k>0} (a_k)^2(u_{k+1}-u_k)\ln (F((\xi_{k+1}-v_k)/a_k)-F((\xi_k-v_k)/a_k))+\nonumber\\
\frac{1}{2}\sum_{k=0,\ldots,n-1,a_k=0}(u_{k+1}-u_k)(\xi_{k+1}-v_k)^2, \quad \bar\xi=(\xi_1,\ldots,\xi_d)\in\Omega.
\end{align}
We will call this function \textbf{the entropy} because it depends only on the discontinuities of a solution.
Thus, for $\bar\xi\in\Int\Omega$ the e.s. (\ref{sol}) corresponds to a critical point of the entropy. We are going to demonstrate that the entropy is strictly convex and coercive in $\Omega$. Therefore, it has a unique global minimum point in $\Omega$. In the case when this minimum point lies in $\Int\Omega$ it is a unique critical point.

Obviously, $E(\bar\xi)\in C^\infty(\Omega)$. Notice that for all $k=0,\ldots,n-1$, such that $a_k>0$
$$
\ln (F((\xi_{k+1}-v_k)/a_k)-F((\xi_k-v_k)/a_k))<0.
$$
Therefore, all terms in expression (\ref{E}) are nonnegative and, in particular, $E(\bar\xi)\ge 0$.

\begin{proposition}[coercivity]\label{pro1}
The sets $E(\bar\xi)\le c$ are compact for each constant $c\ge 0$.
\end{proposition}

\begin{proof}
If $E(\bar\xi)\le c$ then it follows from nonnegativity of all terms in (\ref{E}) that for all $k=0,\ldots,n-1$
\begin{align}\label{3a}
-(a_k)^2(u_{k+1}-u_k)\ln (F((\xi_{k+1}-v_k)/a_k)-F((\xi_k-v_k)/a_k))\le E(\bar\xi)\le c \mbox{ if } a_k>0, \\
\label{3b}
(u_{k+1}-u_k)(\xi_{k+1}-v_k)^2/2\le c \ \mbox{ if } a_k=0.
\end{align}
Relation (\ref{3a}) implies the estimate
\begin{equation}\label{4}
F((\xi_{k+1}-v_k)/a_k)-F((\xi_k-v_k)/a_k)\ge \delta\doteq\exp(-c/m)>0,
\end{equation}
where $\displaystyle m=\min_{k=0,\ldots,n-1, a_k>0}(a_k)^2(u_{k+1}-u_k)>0$.
If $a_0>0$ relation (\ref{4}) with $k=0$ reads $F((\xi_1-v_0)/a_0)>\delta$ (notice that $F((\xi_0-v_0)/a_0)=F(-\infty)=0$), which implies that
$$
\xi_1\ge v_0+a_0F^{-1}(\delta).
$$
On the other hand, if $a_0=0$ then $(u_1-u_0)(\xi_1-v_0)^2\le 2c$, in view of (\ref{3b}) with $k=0$, and
$$
\xi_1\ge v_0-(2c/(u_1-u_0))^{1/2}.
$$
In any case,
\begin{equation}\label{lb}
\xi_1\ge r_1\doteq v_0+\min(a_0F^{-1}(\delta),-(2c/(u_1-u_0))^{1/2}).
\end{equation}
To get an upper bound, we remark that in the case $a_{n-1}>0$ it follows from (\ref{4}) with $k=d=n-1$ that
$F(-(\xi_{n-1}-v_{n-1})/a_{n-1})=1-F((\xi_{n-1}-v_{n-1})/a_{n-1})\ge \delta$ (observe that $F((\xi_n-v_{n-1})/a_{n-1})=F(+\infty)=1$), which implies the estimate
$$
\xi_d\le v_{n-1}-a_{n-1}F^{-1}(\delta).
$$
If $a_{n-1}=0$ then $d=n$ and in view of inequality (\ref{3b}) with $k=n-1$ we find $(u_n-u_{n-1})(\xi_n-v_{n-1})^2/2\le c$, that is,
$$
\xi_d\le v_{n-1}+(2c/(u_n-u_{n-1}))^{1/2}.
$$
In both cases
\begin{equation}\label{ub}
\xi_d\le r_2\doteq v_{n-1}+\max(-a_{n-1}F^{-1}(\delta),(2c/(u_n-u_{n-1}))^{1/2}).
\end{equation}
Since all coordinates of $\bar\xi$ lie between $\xi_1$ and $\xi_d$, estimates (\ref{lb}), (\ref{ub}) imply the bound
$$
|\bar\xi|_\infty=\max_{k=1,\ldots,d} |\xi_k|\le r\doteq\max(|r_1|,|r_2|).
$$
Further, since $F'(x)=\frac{1}{\sqrt{2\pi}}e^{-x^2/2}<1$, the function $F(x)$ is Lipschitz with constant $1$ and it follows from (\ref{4}) that
$$
(\xi_{k+1}-\xi_k)/a_k\ge F((\xi_{k+1}-v_k)/a_k)-F((\xi_k-v_k)/a_k)\ge \delta, \quad k=1,\ldots,d-1, a_k>0.
$$
We find that $$\xi_{k+1}-\xi_k\ge a_k\delta$$
(this also includes the case $a_k=0$).
We conclude tat the set $E(\bar\xi)\le c$ lies in the compact set
$$
K=\{ \ \bar\xi=(\xi_1,\ldots,\xi_d)\in\R^d \ | \ |\bar\xi|_\infty\le r, \ \xi_{k+1}-\xi_k\ge a_k\delta \ \forall k=1 ,\ldots,d-1 \ \}.
$$
By the continuity of $E(\bar\xi)$ the set $E(\bar\xi)\le c$ is a closed subset of $K$ and therefore is compact.
\end{proof}
We take $c>N\doteq\inf E(\bar\xi)$. Then the set $E(\bar\xi)\le c$ is not empty. By Proposition~\ref{pro1} this set is compact and therefore the continuous function $E(\bar\xi)$ reaches the minimal value on it, which is evidently equal to $N$. We proved the existence of global minimum $E(\bar\xi_0)=\min E(\bar\xi)$. The uniqueness of the minimum point is a consequence of strict convexity of the entropy, which is stated in Proposition~\ref{pro2} below. The following lemma plays a key role.

\begin{lemma}\label{lem2}
The function $P(x,y)=-\ln (F(x)-F(y))$ is strictly convex in the half-plane $x>y$.
\end{lemma}

\begin{proof}
The function $P(x,y)$ is infinitely differentiable in the domain $x>y$. To prove the lemma, we need to establish that the Hessian $D^2 P$ is positive definite at every point. By the direct computation we find
\begin{align*}
\frac{\partial^2}{\partial x^2} P(x,y)=\frac{(F'(x))^2-F''(x)(F(x)-F(y))}{(F(x)-F(y))^2}, \\
\frac{\partial^2}{\partial y^2} P(x,y)=\frac{(F'(y))^2-F''(y)(F(y)-F(x))}{(F(x)-F(y))^2}, \
\frac{\partial^2}{\partial x\partial y} P(x,y)=-\frac{F'(x)F'(y)}{(F(x)-F(y))^2}.
\end{align*}
We have to prove positive definiteness of the matrix $Q=(F(x)-F(y))^2 D^2 P(x,y)$ with the components
\begin{align*}
Q_{11}=(F'(x))^2-F''(x)(F(x)-F(y)), \\ Q_{22}=(F'(y))^2-F''(y)(F(y)-F(x)), \ Q_{12}=Q_{21}=-F'(x)F'(y).
\end{align*}
Since $F'(x)=e^{-x^2/2}$, then $F''(x)=-xF'(x)$ and the diagonal elements of this matrix can be written in the form
\begin{align*}
Q_{11}=F'(x)(x(F(x)-F(y))+F'(x))= \\ F'(x)(x(F(x)-F(y))+(F'(x)-F'(y)))+F'(x)F'(y), \\
Q_{22}=F'(y)(y(F(y)-F(x))+(F'(y)-F'(x)))+F'(x)F'(y).
\end{align*}
By Cauchy mean value theorem there exists such a value $z\in (y,x)$ that
$$
\frac{F'(x)-F'(y)}{F(x)-F(y)}=\frac{F''(z)}{F'(z)}=-z.
$$
Therefore,
\begin{align*}
Q_{11}=F'(x)(F(x)-F(y))(x-z)+F'(x)F'(y), \\ Q_{22}=F'(y)(F(x)-F(y))(z-y)+F'(x)F'(y),
\end{align*}
and it follows that $Q=R_1+F'(x)F'(y)R_2$, where $R_1$ is a diagonal matrix with the positive diagonal elements
$F'(x)(F(x)-F(y))(x-z)$, $F'(y)(F(x)-F(y))(z-y)$ while $R_2=\left(\begin{smallmatrix} 1 & -1 \\ -1 & 1\end{smallmatrix}\right)$. Since $R_1>0$, $R_2\ge 0$, then the matrix $Q>0$, as was to be proved.
\end{proof}

\begin{corollary}\label{cor1}
The functions $P(x,-\infty)=-\ln F(x)$, $P(+\infty,x)=-\ln(1-F(x))$ of single variable are strictly convex.
\end{corollary}

\begin{proof}
Since $1-F(x)=F(-x)$, we see that $P(+\infty,x)=P(-x,-\infty)$, and it is sufficient to prove the strict convexity of the function
$P(x,-\infty)=-\ln F(x)$. By Lemma~\ref{lem2} in the limit as $y\to-\infty$ we obtain that this function is convex, moreover,
$$
0\le (F(x))^2\frac{d^2}{dx^2}P(x,-\infty)=\lim_{y\to-\infty}Q_{11}=F'(x)(xF(x)+F'(x)).
$$
If $\frac{d^2}{dx^2}P(x,-\infty)=0$ at some point $x=x_0$ then $0=x_0F(x_0)+F'(x_0)$ is the minimum of the nonnegative function $xF(x)+F'(x)$. Therefore, its derivative $(xF+F')'(x_0)=0$. Since $F''(x)=-xF'(x)$, this derivative
$$
(xF+F')'(x_0)=F(x_0)+x_0F'(x_0)+F''(x_0)=F(x_0)>0.
$$
But this contradicts our assumption. We conclude that $\frac{d^2}{dx^2}P(x,-\infty)>0$ and the function $P(x,-\infty)$ is strictly convex.
\end{proof}

\begin{proposition}[convexity]\label{pro2}
The entropy function $E(\bar\xi)$ is strictly convex on $\Omega$.
\end{proposition}

\begin{proof}
For $k=0,\ldots,n-1$ we denote $P_k(\bar\xi)=-\ln (F((\xi_{k+1}-v_k)/a_k)-F((\xi_k-v_k)/a_k))$ if $a_k>0$, 
and $P_k(\bar\xi)=(\xi_{k+1}-v_k)^2$ if $a_k=0$. In view of (\ref{E}) the entropy $E(\bar\xi)$ is a linear combination of the functions $P_k$ with positive coefficients, and convexity of the entropy readily follows from the statements of
Lemma~\ref{lem2} and Corollary~\ref{cor1}. To establish the strict convexity, we have to demonstrate that the Hessian matrix $D^2 E(\bar\xi)$ is strictly positive. Assume that for some $\zeta=(\zeta_1,\ldots,\zeta_d)\in\R^d$
\begin{equation}\label{form}
D^2 E(\bar\xi)\zeta\cdot\zeta=\sum_{i,j=1}^d\frac{\partial^2 E(\bar\xi)}{\partial\xi_i\partial\xi_j}\zeta_i\zeta_j=0.
\end{equation}
Since $E(\bar\xi)$ is a linear combination of convex functions $P_k(\bar\xi)$ with positive coefficients, we find that 
$$
D^2 P_k(\bar\xi)\zeta\cdot\zeta=0 \quad \forall k=0,\ldots,n-1.
$$
This can be written in the form
\begin{align*}
\sum_{i,j=k,k+1}\frac{\partial^2 P_k(\bar\xi)}{\partial\xi_i\partial\xi_j}\zeta_i\zeta_j=0 \ \mbox{ if } 0<k<n-1, a_k>0;\\
\frac{\partial^2 P_k(\bar\xi)}{\partial\xi_{k+1}^2}\zeta_{k+1}^2 \ \mbox{ if } k=0 \mbox{ or } a_k=0.
\end{align*}
In view of Lemma~\ref{lem2} and Corollary~\ref{cor1} the functions $P_k$ in above equalities are strictly convex as functions of either two variables
$(\xi_k,\xi_{k+1})$ or single variable $\xi_{k+1}$. Therefore, these equalities imply that in any case $\zeta_{k+1}=0$, $k=0,\ldots,n-2$, and  $\zeta_n=0$ if $a_{n-1}=0$ (when $d=n$). We conclude that all coordinates $\zeta_i=0$, $i=1,\ldots,d$. Hence, equality (\ref{form}) can hold only for $\zeta=0$ and the matrix $D^2 P(\bar\xi)>0$ for all $\bar\xi\in\Omega$. This completes the proof.
\end{proof}

\section{The variational formulation}

Let $\bar\xi_0=(\xi_1,\ldots,\xi_d)\in\Omega$ be the unique minimum point of $E(\bar\xi)$. The necessary and sufficient condition for $\bar\xi_0$ to be a minimum point is the following one
\begin{equation}\label{extr}
\nabla E(\bar\xi_0)\cdot p\ge 0 \quad \forall p\in T(\bar\xi_0)=\{ \ p\in\R^d \ | \ \exists h>0 \ \bar\xi_0+hp\in \Omega \ \},
\end{equation}
so that $T(\bar\xi_0)$ is the tangent cone to $\Omega$ at the point $\bar\xi_0$. If $\bar\xi_0\in\Int\Omega$ then
$T(\bar\xi_0)=\R^d$ and (\ref{extr}) reduces to the requirement $\nabla E(\bar\xi_0)=0$. As we have already demonstrated, this requirement coincides with jump conditions (\ref{c++}), (\ref{c+0}), (\ref{c0+}), (\ref{c00}) for all $k=1,\ldots,d$.
But these conditions are equivalent to the statement that the function (\ref{sol}) is an e.s. of (\ref{1}), (\ref{2}).
In the general situation when $\bar\xi_0$ can belong to the boundary of $\Omega$, the coordinates of $\bar\xi_0$ may coincides. Let $\xi_k=\cdots=\xi_l=c$ be a maximal family of coinciding coordinates, that is, $\xi_{k-1}<\xi_k=\xi_l<\xi_{l+1}$ (it is possible here that $k=l$). Then, as is easy to realize, the vector $p=(p_1,\ldots,p_d)$, with arbitrary increasing coordinates $p_k\le\cdots\le p_l$ and with zero remaining coordinates, belong to the tangent cone $T(\bar\xi_0)$. In view of (\ref{extr})
$$
\sum_{i=k}^l\frac{\partial}{\partial\xi_i} E(\bar\xi_0)p_i\ge 0
$$
for any such a vector. Using the summation by parts formula, we realize that the above condition is equivalent to the following requirements
\begin{align}\label{4a}
\sum_{i=k}^l\frac{\partial}{\partial\xi_i} E(\bar\xi_0)=0, \\
\label{4b}
\sum_{i=j}^l\frac{\partial}{\partial\xi_i} E(\bar\xi_0)\ge 0, \ k<j\le l.
\end{align}
Recall that $a_i=0$ for $k\le i<l$. By the direct computation we find
\begin{align*}
\frac{\partial}{\partial\xi_i} E(\bar\xi_0)=(u_i-u_{i-1})(\xi_i-v_{i-1}), \quad k<i<l, \\
\frac{\partial}{\partial\xi_k} E(\bar\xi_0)=\left\{\begin{array}{lcr} (u_k-u_{k-1})(\xi_k-v_{k-1}) & , & a_{k-1}=0, \\
-A(u)'(c-) & , & a_{k-1}>0; \end{array}\right. \\
\frac{\partial}{\partial\xi_l} E(\bar\xi_0)=A(u)'(c+),
\end{align*}
where $A(u)'(c\pm)$ are given by (\ref{der-}), (\ref{der+}). Putting these expressions into (\ref{4a}), (\ref{4b}), we obtain exactly the jump conditions (\ref{RGb}), (\ref{Olb}).
Therefore, the function (\ref{sol}) corresponding to the point $\bar\xi_0$ is an e.s. of (\ref{1}), (\ref{2}).
Conversely, if (\ref{sol}) is an e.s. then relations (\ref{4a}), (\ref{4b}) holds for all groups of coinciding
coordinates. As is easy to verify, this is equivalent to the criterion (\ref{extr}). We have proved our main result.

\begin{theorem}\label{th1}
The function (\ref{sol}) is an e.s. of (\ref{1}), (\ref{2}) if and only if $\bar\xi_0=(\xi_1,\ldots,\xi_d)$ is the minimum point of the entropy $E(\bar\xi)$.
\end{theorem}

\begin{remark}\label{rem1}
Adding to the entropy (\ref{E}) the constant $$\sum_{k=0,\ldots,n-1, a_k>0} (a_k)^2(u_{k+1}-u_k)\ln ((u_{k+1}-u_k)/a_k),$$ we obtain the alternative variant of the entropy
\begin{align}\label{E1}
E_1(\bar\xi)=-\sum_{k=0,\ldots,n-1, a_k>0} (a_k)^2(u_{k+1}-u_k)\ln\left(\frac{F((\xi_{k+1}-v_k)/a_k)-F((\xi_k-v_k)/a_k)}{(u_{k+1}-u_k)/a_k}\right) \nonumber\\ +\frac{1}{2}\sum_{k=0,\ldots,n-1, a_k=0} (u_{k+1}-u_k)(\xi_{k+1}-v_k)^2.
\end{align}
If we consider the values $v_k,a_k$ as a piecewise constant approximation of an arbitrary velocity function $v(u)$ and, respectively, a diffusion function $a(u)\ge 0$ then, passing in (\ref{E1}) to the limit as $\max(u_{k+1}-u_k)\to 0$, we find that the entropy $E_1(\bar\xi)$ turns into the variational functional
\begin{align*}
J(\xi)=-\int_{\{u\in [\alpha,\beta], a(u)>0\}} (a(u))^2\ln(F'((\xi(u)-v(u))/a(u))\xi'(u))du+ \\ \frac{1}{2}\int_{\{u\in [\alpha,\beta], a(u)=0\}} (\xi(u)-v(u))^2du,
\end{align*}
where $\xi(u)$ is an increasing function on $[\alpha,\beta]$, which is expected to be the inverse function to a self-similar solution $u=u(\xi)$ of the problem (\ref{1}), (\ref{2}). Taking into account that
\begin{align*}
\ln(F'((\xi(u)-v(u))/a(u))\xi'(u))=\ln F'((\xi(u)-v(u))/a(u))+\ln\xi'(u)= \\ -\frac{(\xi(u)-v(u))^2}{2a^2(u)}+\ln\xi'(u),
\end{align*}
we may simplify the expression for the functional $J(\xi)$
\begin{equation}\label{var}
J(\xi)=\int_\alpha^\beta [(\xi(u)-v(u))^2/2-(a(u))^2\ln(\xi'(u))]du.
\end{equation}
We see that this functional is strictly convex. The corresponding Euler-Lagrange equation has the form
\begin{equation}\label{EL}
\xi(u)-v(u)+((a(u))^2/\xi'(u))'=0.
\end{equation}
Since $u'(\xi)=1/\xi'(u)$, $u=u(\xi)$, we can transform (\ref{EL}) as follows $$\xi(u)-v(u)+((a(u))^2 u'(\xi))'_u=0.$$ Multiplying this equation by $u'(\xi)$, we obtain the equation
$$
(a^2u')'=(v-\xi)u', \quad u=u(\xi),
$$
which is exactly our equation (\ref{1}) written in the self-similar variable.
\end{remark}

\begin{remark}\label{rem2}
In the case of conservation laws (\ref{con}) the e.s. $u=u(\xi)$ of (\ref{con}), (\ref{2}) is piecewise constant, and, by expression (\ref{sol}),
$$
u(\xi)=u_k, \quad \xi_k<\xi<\xi_{k+1}, \ k=0,\ldots,n,
$$
where $-\infty=\xi_0<\xi_1\le\cdots\le\xi_n<\xi_{n+1}=+\infty$. In this case the entropy function is particularly simple, it is the quadratic function
$$
E(\bar\xi)=\frac{1}{2}\sum_{k=1}^n (u_k-u_{k-1})(\xi_k-v_{k-1})^2,
$$
defined on the closed polyhedral cone
$$
\Omega=\{ \ \bar\xi=(\xi_1,\ldots,\xi_n)\in\R^n \ | \ \xi_{k+1}\ge\xi_k \ \forall k=1,\ldots,n-1 \ \}.
$$
Existence and uniqueness of a minimal point in this case is trivial. By Theorem~\ref{1} and Remark~\ref{rem1} we obtain new, variational formulation of the entropy solution.
\end{remark}

\section*{Acknowledgments}
The research was supported by the Russian Science Foundation, grant 22-21-00344.


\begin{thebibliography}{9}
\bibitem{Car}
J.~Carrillo, Entropy solutions for nonlinear degenerate problems, Arch. Ration. Mech. Anal.,
147 (1999), 269--361.
\bibitem{Kr}
S.\,N.~Kruzhkov,
First order quasilinear equations in several independent variables, Mat. Sb. (N.S.), 81 (1970), 228--255.
\bibitem{Ol} O.\,A.~Oleinik, Uniqueness and stability of the generalized solution of the Cauchy problem for a quasi-linear
equation, Uspekhi Mat. Nauk, 14:2(86) (1959), 165--170.
\end{thebibliography}
\end{document}